\documentclass[preprint]{elsarticle}

\usepackage{hyperref}

\usepackage{amsmath}
\usepackage{amsfonts}
\usepackage{amssymb}
\usepackage{graphicx}
\usepackage{comment}
\usepackage{subcaption}
\usepackage[ruled]{algorithm2e}
\usepackage{algorithmic}
\usepackage{doi,url}

\usepackage{exscale}

\usepackage{color}

\newtheorem{theorem}{Theorem}

\newtheorem{proposition}{Proposition}
\newtheorem{corollary}{Corollary}

\newcommand{\eq}[1]{\begin{equation}\label{#1}}
\newcommand{\en}{\end{equation}}

\newenvironment{proof}{\begin{trivlist}
                       \item[]{\bf Proof.}
                       \hspace{0cm} }{\hfill $\Box$
                       \end{trivlist}}

\begin{document}

\begin{frontmatter}

\title{Preconditioned steepest descent-like methods for symmetric indefinite systems
\tnoteref{mytitlenote2}}
\tnotetext[mytitlenote2]{Results are partially based on  
PhD thesis~\cite{thesis} of the first coauthor.}

%
%
%
%

\author{Eugene Vecharynski\corref{mycorrespondingauthor}}
\cortext[mycorrespondingauthor]{Corresponding author}
\ead{eugene.vecharynski@gmail.com}
\ead[url]{http://evecharynski.com/}
\address{Computational Research Division, Lawrence Berkeley National Laboratory; 1 Cyclotron Road, Berkeley, CA 94720, USA}

\author{Andrew Knyazev}
\ead{Andrew.Knyazev@merl.com}
\ead[url]{{http://www.merl.com/people/knyazev}}
\address{Mitsubishi Electric Research Laboratories; 201 Broadway Cambridge, MA 02139, USA}

%
%

\begin{abstract}
This paper addresses the question of what exactly is an analogue of 
the preconditioned steepest descent (PSD) algorithm in the case of a symmetric indefinite system
with an SPD preconditioner. We show that a basic PSD-like scheme for an SPD-preconditioned 
symmetric indefinite system is mathematically equivalent to the restarted PMINRES, 
where restarts occur after every two steps.
A convergence bound is derived. If certain information on the spectrum
of the preconditioned system is available, we present a simpler PSD-like algorithm that performs  
only one-dimensional residual minimization. 
%
Our primary goal is to bridge the theoretical gap between optimal (PMINRES) and 
PSD-like methods for solving symmetric indefinite systems, as well as point 
out situations where the PSD-like schemes can be used in practice.

\end{abstract}


\end{frontmatter}


\section{Introduction}\label{sec:intro}

The Preconditioned Steepest Descent (PSD) iteration is a well known precursor of the optimal 
Preconditioned Conjugate Gradient (PCG) algorithm for solving Symmetric Positive Definite (SPD) linear systems.
Given a system~$Ax = f$ with an SPD matrix $A$ and an SPD preconditioner $T$ 
the method at each iteration $i$ updates the current approximate solution $x^{(i)}$ as  
\begin{equation}\label{eq:psd}
x^{(i+1)} = x^{(i)} + \alpha^{(i)} T (f - A x^{(i)}), \quad i = 0, 1, \dots; 
\end{equation}
where the iterative parameter $\alpha^{(i)}$ 
is chosen to ensure that the new approximation $x^{(i+1)}$ has the smallest, 
among all vectors of the form $x +  \alpha T(f - Ax)$, $A$-norm of the error $x^{(i+1)}-x$.  

The optimality of PCG stems from its ability to construct approximations~$x^{(i)}$
that globally minimize the $A$-norm of the error over an expanding sequence of Krylov subspaces while relying on 
a short-term recurrence~\cite{Axelsson:94, Greenbaum:97}. In contrast, the PSD iteration~\eqref{eq:psd}
is \textit{locally} optimal, searching for a best approximation $x^{(i+1)}$ only in a single direction,  
given by the preconditioned residual $T(f - Ax^{(i)})$.   


The lack of global optimality in PSD leads to a lower convergence rate. In particular,
instead of the asymptotic convergence factor $(\sqrt{\kappa} - 1)/(\sqrt{\kappa} + 1)$, guaranteed by 
the optimal PCG, each PSD step is guaranteed to reduce the error $A$-norm 
by the factor $(\kappa - 1)/(\kappa + 1)$, e.g.,\ ~\cite{Axelsson:94, Greenbaum:97}, 
and the error Euclidean norm by the factor $1-1/\kappa$, see \cite{Knyazev1988195}, where $\kappa$ 
denotes a spectral condition number of the preconditioned matrix $TA$. 
Nevertheless, despite its generally slower 
convergence, PSD (and even simpler iterations, such as Jacobi or Gauss-Seidel) 
finds its way to practical applications, due to a reduced amount of memory and  
computations
per iteration~\cite{Nagy.Palmer:03, Knyazev.Lashuk:07, Trottenberg.Oosterlee.Schuller:01}.


If the matrix $A$ is symmetric \textit{indefinite}, then an optimal analogue of PCG
is given by the preconditioned MINRES (PMINRES) algorithm~\cite{Paige.Saunders:75, Choi.Paige.Saunders:11}
\footnote{PMINRES is mathematically equivalent to preconditioned 
Orthomin(2) and Orthodir(3) algorithms (e.g.,~\cite{Greenbaum:97}) that can as well be viewed as optimal analogues of PCG for 
symmetric indefinite systems. 
However, Orthomin(2) can break down, whereas Orthodir(3) has 
a higher computational cost compared to PMINRES. Therefore, throughout, 
we do not discuss these two alternative schemes, 
and consider only the PMINRES algorithm.}.
Similar to PCG, PMINRES utilizes a short-term recurrence to achieve optimality 
with respect to the expanding sequence of the Krylov subspaces~\cite{Greenbaum:97, Elman.Silvester.Wathen}. 
However, since $A$ is indefinite, minimization of the error $A$-norm is no longer feasible. 
Instead, PMINRES minimizes the $T$-norm of the residual $f - Ax^{(i+1)}$, where $T$ is a given SPD preconditioner. 

The symmetry and positive definiteness of the preconditioner is generally critical for PMINRES.
Under this assumption the method is guaranteed to converge, with the convergence bound
described in terms of the spectrum 
$$\Lambda(TA) = \left\{\lambda_1 \leq \ldots \leq \lambda_p \leq \lambda_{p+1} \leq \ldots \leq \lambda_n \right\}$$ 
of the preconditioned matrix $TA$. In particular, assuming that
$\Lambda(TA)$ is located within the union of two equal-sized intervals $\mathcal{I} = \left[a,b\right] \bigcup \left[c,d\right]$,
where $a \leq \lambda_1 \leq \lambda_p \leq b < 0 < c \leq \lambda_{p+1} \leq \lambda_n \leq d$, the following
bound on the residual $T$-norm holds: 
\begin{equation}\label{eq:minres_cv}
\|r^{(i)}\|_T \leq 2 \left( \frac{\sqrt{|ad|} - \sqrt{|bc|}}{\sqrt{|ad|} + \sqrt{|bc|}} \right)^{[i/2]} \|r^{(0)}\|_T,
\end{equation}
where $i$ is the PMINRES iteration number and 
$\| \cdot \|_T \equiv (\cdot \ , \ T \ \cdot)^{1/2}$~\cite{Greenbaum:97, Elman.Silvester.Wathen}.

While the optimal PMINRES algorithm is used in a variety of applications and has convergence behavior that 
is relatively well studied, to the best of our knowledge, little or none has been said 
about PSD-like methods for symmetric indefinite systems, where the preconditioner is SPD, 
i.e., is exactly the same as in PMINRES. 
For example, as we explain in the next section, iterations of the form of~\eqref{eq:psd} cannot generally result in a convergent scheme. 

In this paper we address the question of what exactly is an analogue of PSD in the case of a symmetric indefinite system
with an SPD preconditioner.~In~particular, exactly the same way PSD can be interpreted as a form of PCG restarted at 
every step, we show that a basic PSD-like scheme for an SPD-preconditioned symmetric indefinite 
system is mathematically equivalent to the restarted PMINRES, where restarts occur after every two steps,
i.e., the residual $T$-norm is minimized over two-dimensional subspaces. We derive a convergence bound, 
which yields a stepwise convergence factor that is similar 
to the~one in~\eqref{eq:minres_cv}~up~to the presence of square roots, analogously to the
PCG/PSD case for SPD~systems.    

We also demonstrate that,
if certain information about the spectrum of~the preconditioned matrix is at hand, 
then the two-dimensional
minimization can~be turned into minimization over a one-dimensional subspace, 
while guaranteeing the same convergence bound. 
Such information can also provide an interesting possibility for randomization of the descent direction, 
which we as well~briefly discuss in this~paper.


Although the primary goal of this work is to bridge the theoretical gap between optimal 
(PMINRES) and PSD-like methods for solving symmetric indefinite systems,
we also address several practical issues. In particular, we discuss implementations 
of the PSD-like algorithms, which should be fulfilled carefully in order to ensure
a minimal amount of computation and storage per~iteration.

Because of the inferior convergence rate, the PSD-like methods cannot 
be generally regarded as an alternative to the optimal PMINRES.
However, we point out several specific situations where the use of the more economical PSD-like
iterations is appropriate and can be preferred in practice.
Such situations arise, e.g., when only a few iterations of a linear solver are needed,
due to a high preconditioning quality, good initial guess, or a relaxed requirement
on the accuracy of the approximate solution. 
%
For example, this setting appears in the framework of preconditioned interior eigenvalue 
calculations, where a preconditioner can be defined
by several steps of a linear solver applied to a shifted system of the form 
$(A - \sigma B)w=r$~\cite{Szyld.Ve.Xue:15, Ve.Yang.Xue:15, Cai.Bai.Pask.Sukumar:13}. 
The PSD-like methods can also be used as smoothers in multigrid 
schemes~\cite{Briggs.Henson.McCormick:00,Trottenberg.Oosterlee.Schuller:01}.
In any of these contexts, the savings in storage and number of inner products offered 
by the PSD-like algorithms can potentially be beneficial for achieving the best performance.

The paper is organized as follows. In Section \ref{sec:psi}, we present a basic form of the 
PSD-like iteration for solving a symmetric indefinite system with an SPD preconditioner, which is 
based on two-dimensional minimization of the residual $T$-norm, and derive the  
convergence bound. In Section~\ref{sec:prmm}, we show~how some knowledge of spectrum of the preconditioned matrix
can simplify the~PSD-like iteration, leading to a scheme which minimizes the residual over a one-dimensional
subspace. A simple randomization strategy is described in the~same~section. 
We consider several examples in Section~\ref{sec:num}. Conclusions can be found
in Section~\ref{sec:concl}.

\section{The PSD-like iteration for symmetric indefinite systems}\label{sec:psi}

Given an SPD preconditioner $T$, a candidate PSD-like scheme for
symmetric indefinite systems can be immediately defined by directly applying iterations of the form~\eqref{eq:psd}.   
%
In this case,
the corresponding error equation has the~form
\begin{equation}\label{eq:error}
e^{(i+1)} = (I - \alpha^{(i)} TA)e^{(i)},
\end{equation}   
where $e^{(i)} = x^* - x^{(i)}$ is the error at step $i$ and $x^{*} = A^{-1}f$ is the exact solution.

Let $y_j$ be the eigenvectors of the preconditioned matrix $TA$ associated with the eigenvalues $\lambda_j$,
and suppose that $e^{(i)} = \sum_{j=1}^n c_j y_j$ represents an expansion of error in the eigenvector basis 
with coefficients~$c_j$. Then, according to~\eqref{eq:error},
\begin{equation}\label{eq:error_component}
e^{(i+1)} = \sum_{j=1}^n (1 - \alpha^{(i)} \lambda_j) c_j y_j.
\end{equation} 
Since $\Lambda(TA)$ contains both positive and negative eigenvalues, 
for any choice of the iteration parameter $\alpha^{(i)}$, there exist $\lambda_j$'s of an opposite sign, 
i.e.,  such that $\text{sign}(\lambda_{j}) = - \text{sign}(\alpha^{(i)})$. 
In this case, the corresponding factors 
$\mu_j \equiv 1 - \alpha^{(i)} \lambda_j = 1 + |\alpha^{(i)} \lambda_j|$
in~\eqref{eq:error_component} are greater than one. 

Thus, regardless of the choice of $\alpha^{(i)}$, when applied to a symmetric indefinite system with an SPD
preconditioner, iteration~\eqref{eq:psd} will amplify the error in certain directions. Hence, it does not
deliver a convergent scheme, unless initial guess is specially chosen. Therefore, we cannot 
consider~\eqref{eq:psd} as an analogue of PSD in the indefinite case.    


A possible angle to look at~\eqref{eq:psd} is as to a restarted Krylov subspace method. 
In particular, the PSD algorithm for SPD systems
can be interpreted as PCG that is restarted at each step. 
The same viewpoint can be adopted for systems with an indefinite $A$ and an SPD $T$. 
In this case, we can define an analogue of PSD as a properly restarted version of PMINRES. 
As shown above, restarting PMINRES at every step\footnote{Such a scheme is equivalent to 
preconditioned Orthomin(1); see, e.g.,~\cite{Greenbaum:97}.}, 
which yields iteration of the form~\eqref{eq:psd}, fails to ensure the convergence. 
Therefore, we are interested in determining the frequency of restarts which,
on the one hand, keeps the size of the local minimization subspace as small as possible and, on the other hand, 
guarantees the convergence.    

Following these considerations, 
it is natural to consider an iterative scheme that is obtained from PMINRES by restarting 
the method after 
\textit{every two steps}. This gives iteration of the form
\begin{equation}\label{eq:psd_indef}
x^{(i+1)} = x^{(i)} + \beta^{(i)} T r^{(i)} + \alpha^{(i)} TAT r^{(i)} , \quad i = 0, 1, \dots; 
\end{equation}
where the parameters $\alpha^{(i)}$ and $\beta^{(i)}$ are chosen to minimize the residual $T$-norm, i.e.,
are such that 
\begin{equation}\label{eqn:pm2_opt_cond}
\|r^{(i+1)}\|_T = \min_{u \in \text{span}\left\{ T r^{(i)}, TATr^{(i)}\right\}} \|r^{(i)} - Au\|_T. 
\end{equation}
In what follows, we prove that~\eqref{eq:psd_indef}--\eqref{eqn:pm2_opt_cond} converges 
at a linear rate that is similar to that of PSD and, hence, represents a true analogue 
of PSD for 
symmetric indefinite systems.    

\subsection{The convergence bound}

Let us first consider
a stationary iteration of the form
\begin{equation}\label{eqn:psi_sid}
\begin{array}{ccl}
r^{(i)}   & = & f - A x^{(i)}, \ w^{(i)} = T r^{(i)}, \ s^{(i)} = TA w^{(i)}, \ l^{(i)}  =  s^{(i)} - \beta w^{(i)}, \\
x^{(i+1)} & = & x^{(i)} + \alpha l^{(i)}, \quad i = 0,1,\ldots,
\end{array}
\end{equation}
where the parameters $\alpha > 0$ and $\beta$ remain constant at all steps. 
Scheme~\eqref{eqn:psi_sid} can be viewed as a preconditioned Richardson-like method~\cite{Axelsson:94} 
with the search direction given by $l^{(i)}$, which is a linear combination of 
$w^{(i)} = Tr^{(i)}$ and $s^{(i)} = TAw^{(i)}$. 
The following theorem specifies the values of 
$\alpha$ and $\beta$ that yield the convergence of~\eqref{eqn:psi_sid}, 
and states the corresponding convergence bound.
\begin{theorem}\label{thm:psi_cv}
Let iterations~(\ref{eqn:psi_sid}) be applied to a system $Ax = f$
with a nonsingular symmetric indefinite $A$ and an SPD preconditioner $T$, and
assume that the spectrum of $TA$ is enclosed within the pair of intervals 
$\mathcal{I} = [a,b] \bigcup [c,d]$ of equal length.
If $b < \beta < c$ and $0 < \alpha < \tau_{\beta}$, 
where $\tau_{\beta} = 2/\displaystyle \max_{\lambda \in \left\{a,d\right\}}(\lambda^2 - \beta \lambda)$, then 
\begin{equation}\label{eqn:cv}
\frac{\|r^{(i+1)}\|_T}{\|r^{(i)}\|_T} \leq \rho, \ 
\quad \rho = \max_{\lambda \in \left\{a,b,c,d\right\}} \left|1 - \alpha (\lambda^2 - \beta \lambda) \right|<1.
\end{equation}
Moreover, the convergence with optimal factor
\begin{equation}\label{eqn:opt_cv}
\rho \equiv \rho_{opt} = \frac{|ad| - |bc|}{|ad| + |bc|}
\end{equation}
corresponds to the choice $\beta \equiv \beta_{opt} = c - \left| b \right|$ 
and $\alpha \equiv \alpha_{opt} = 2/(| b | c + |a|d)$.
\end{theorem}
\begin{proof}
Let $S_{\beta} = \left( TA  - \beta I \right) TA$. Then the equation for preconditioned residuals of 
iteration~\eqref{eqn:psi_sid} can be written in the form
$Tr^{(i+1)} = ( I - \alpha S_{\beta} ) Tr^{(i)}$,~and 
\[
\|r^{(i+1)}\|_T^2 = ( T^{-1}( I - \alpha S_{\beta} ) Tr^{(i)}, ( I - \alpha S_{\beta} ) Tr^{(i)}) = ( Q y,  Q  y),
\]
where $Q = I - \alpha T^{-1/2} S_{\beta} T^{1/2}$ is a symmetric matrix and $y = T^{1/2} r^{(i)}$. 
Hence, 
\[
\| r^{(i+1)} \|_T^2 \leq  \lambda_{\max}(Q^2) \|r^{(i)}\|_T^2,
\]
where $\lambda_{\max}(Q^2)$ denotes the largest eigenvalue of $Q^2$.  Since $T^{-1/2} S_{\beta} T^{1/2}$ 
is similar to $S_{\beta}$, both matrices have the same eigenvalues 
$\mu_{\beta} (\lambda_j)$, 
where $\mu_{\beta}(\lambda) =  \lambda^2  - \beta \lambda$ and $\lambda_j \in \Lambda(TA)$. 
Thus, 
$$\lambda_{\max}(Q^2) = \max_{\lambda \in \Lambda(TA)} (1 - \alpha \mu_{\beta} (\lambda) )^2
\leq \max_{\lambda \in \mathcal{I}} (1 - \alpha \mu_{\beta} (\lambda) )^2,$$ and
therefore
\begin{equation}\label{eq:res_nrm}
\frac{\| r^{(i+1)}\|_T}{\|r^{(i)}\|_T} \leq 
\rho \equiv \max_{\lambda \in \mathcal{I}} |1 - \alpha \mu_{\beta} (\lambda)| .
\end{equation}

We now determine the values of parameters 
$\alpha$ and $\beta$ that guarantee that $|1 - \alpha \mu_{\beta} (\lambda)| < 1$ for all 
$\lambda \in \mathcal{I}$. Clearly, this is possible only if the value of $\beta$ is
chosen to ensure that $\mu_{\beta} (\lambda)$ is 
of the same sign for all $\lambda \in \mathcal{I}$. 
Therefore, since iteration~\eqref{eqn:psi_sid} assumes that $\alpha > 0$, we require that  
$\beta$'s are such that $\mu_{\beta}(\lambda)$ is positive for all $\lambda \in \mathcal{I}$. 
Since $\mu_{\beta}(\lambda)$ is a parabola, which is concave up with zeros at $0$ and~$\beta$,
$\mu_{\beta}(\lambda) > 0$ on $\mathcal{I}$ if and only if $b < \beta < c$; see Figure~\ref{fig:parab}.

\begin{figure}[h] 
 \begin{center} 	 
 \includegraphics[width = 6cm]{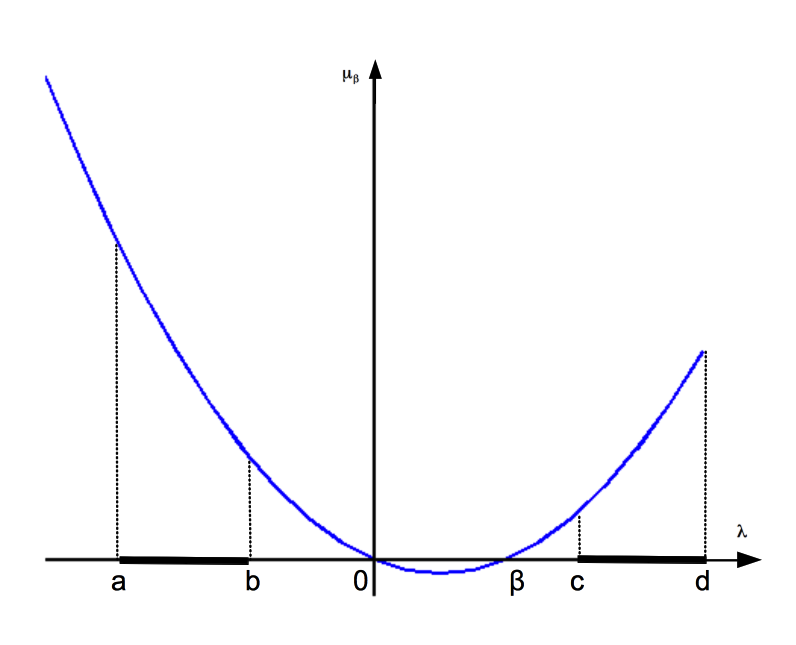} 
 \end{center} 
 \caption{Illustration of $\mu_{\beta}(\lambda) = \lambda^2 - \beta \lambda$, where $\beta > 0$ and  
$\mathcal{I} = [a,b] \cup [c,d]$ .}\label{fig:parab} 
\end{figure}

Given a value $\beta$, 
such that $\mu_{\beta}(\lambda) > 0$ for any $\lambda \in \mathcal{I}$ ($b < \beta < c$), we 
look for parameters $\alpha$ that ensure $\left| 1 - \alpha \mu_{\beta}(\lambda) \right| < 1$. 
Solving this inequality for~$\alpha$ immediately reveals that 
$\left| 1 - \alpha \mu_{\beta}(\lambda) \right|<1$ 
for any $\lambda \in \mathcal{I}$ if 
$0 < \alpha < \tau_{\beta}$, where 
\[
\tau_{\beta} = 2/\displaystyle \max_{\lambda \in \mathcal{I}} \mu_{\beta}(\lambda)
= 2/\displaystyle \max_{\{a,d\}} ( \lambda^2  - \beta \lambda ) , 
\]
with the last equality following from the fact that 
$\mu_{\beta}(\lambda)$ 
attains its maximum on~$\mathcal{I}$ either at $a$ or $d$ (minimum is achieved at $b$ or $c$),
i.e.,
\begin{equation}\label{eq:maxmin_nu}
\max_{\lambda \in \mathcal{I}} \mu_{\beta}(\lambda) = \max_{\lambda \in \{a,d\}} (\lambda^2 - \beta \lambda),\quad 
\min_{\lambda \in \mathcal{I}} \mu_{\beta}(\lambda) = \min_{\lambda \in \{b,c\}} (\lambda^2 - \beta \lambda);
\end{equation}
see Figure~\ref{fig:parab}. 
%
%
%
Thus, for $b < \beta < c$ and $0 < \alpha < \tau_{\beta}$, 
we have $\left| 1 - \alpha \mu_{\beta}(\lambda) \right| < 1$ 
for any $\lambda \in \mathcal{I}$, and therefore the factor 
$\rho$ in~\eqref{eq:res_nrm} is less than $1$. 
Furthermore, the maximum of $|1 - \alpha \mu_{\beta} (\lambda)|$ over $\mathcal{I}$ in~\eqref{eq:res_nrm}
is given either by $| 1 - \alpha \max_{\lambda \in \mathcal{I}} \mu_{\beta} (\lambda) |$
or by $| 1 - \alpha \min_{\lambda \in \mathcal{I}} \mu_{\beta} (\lambda)|$. 
Hence, using~\eqref{eq:maxmin_nu}, we obtain the expression for~$\rho$ as in~\eqref{eqn:cv},     
which 
completes the proof of the first part of the theorem.  

Next, we determine the values of $\alpha$ and $\beta$ that yield the smallest~$\rho$, 
i.e., give an optimal convergence rate. To do so, we first fix an arbitrary $\beta \in (b,c)$
and search for the value of $\alpha$, denoted by $\alpha_{opt} (\beta)$, that minimizes
$\rho$ in~\eqref{eq:res_nrm}. Since, as discussed above, 
$\rho = \max \{ |1 - \alpha \min_{\lambda \in \mathcal{I}} \mu_{\beta}(\lambda)|, |1 - \alpha \max_{\lambda \in \mathcal{I}} \mu_{\beta}(\lambda)| \}$, the optimal value $\alpha_{opt}(\beta)$
is given by 
\begin{equation}\label{eqn:alpha_opt_beta}
\alpha_{opt} (\beta)  =  
 \frac{2}{\displaystyle \min_{\lambda \in \mathcal{I}} \mu_{\beta} (\lambda) + 
 \max_{\lambda \in \mathcal{I}} \mu_{\beta} (\lambda)} =
 \frac{2}{\displaystyle \min_{\lambda \in \{b,c\}} (\lambda^2 - \beta \lambda) + 
 \max_{\lambda \in \{a,d\}} (\lambda^2 - \beta \lambda)}. 
\end{equation}  
For this choice of $\alpha$,  
$
1 - \alpha \min_{\lambda \in \mathcal{I}} \mu_{\beta} (\lambda)  = -( 1 - \alpha \max_{\lambda \in \mathcal{I}} \mu_{\beta} (\lambda) ),
$ 
and, hence, $\rho$ in~\eqref{eq:res_nrm} is given by 
$\rho_{opt}(\beta) = 1 - \alpha_{opt} (\beta) \min_{\lambda \in \mathcal{I}} \mu_{\beta} (\lambda)$.
It is then easy to check, using~\eqref{eq:maxmin_nu} and \eqref{eqn:alpha_opt_beta}, that  
$\rho \equiv \rho_{opt}(\beta)$ can be written in the form
\begin{equation}\label{eq:rho_beta}
\rho_{opt}(\beta) = \frac{\tilde \kappa(\beta) - 1}{\tilde \kappa(\beta) + 1}, \quad 
\tilde \kappa(\beta) = \frac{\displaystyle \max_{\lambda \in \{a,d\}} (\lambda^2  - \beta \lambda)}
{\displaystyle \min_{\lambda \in \{b,c\}} (\lambda^2  - \beta \lambda)}.
\end{equation}
Thus, in order to achieve the smallest $\rho$, it remains to find the value of $\beta$,
denoted by $\beta_{opt}$,
that minimizes $\tilde \kappa(\beta)$ in~\eqref{eq:rho_beta} over all $b < \beta < c$.

Let $\beta \equiv \beta_*= c - \left| b \right|$. In this case, the parabola 
$\mu_{\beta_*}(\lambda)$ is located symmetrically 
with respect to the intervals $[a,b]$ and $[c,d]$. In particular, this implies that 
the largest value of $\mu_{\beta_*}(\lambda) = \lambda^2 - \beta_* \lambda$
is attained simultaneously at $a$ and $d$ and the smallest
value simultaneously occurs at $b$ and $c$. 
Thus, by substituting $\beta_*= c - \left| b \right|$ into $\tilde \kappa (\beta)$ in~\eqref{eq:rho_beta}
and using the assumption that $d - c = |a| - |b|$, 
we obtain
\begin{equation}\label{eq:kappa_beta}
\tilde \kappa(\beta_*) = \frac{d^2 - \beta_{*} d}{c^2 - \beta_{*} c} =  
\displaystyle \left(\frac{d}{c}\right)\left(\frac{|b| + d - c}{|b|}\right) = \frac{ad}{bc}. 
\end{equation}
We now observe that $\beta_{*}$ minimizes $\tilde \kappa (\beta)$ in (\ref{eq:rho_beta}), i.e.,
$\tilde \kappa (\beta_*)$ in~\eqref{eq:kappa_beta} is the smallest for all $\beta$ in $(b,c)$.
Indeed, if $\varepsilon > 0$ is an arbitrary number, then
\[
\tilde \kappa (\beta_* + \varepsilon) = \frac{(a^2 - \beta_* a) - \varepsilon a}{(c^2 - \beta_* c) - \varepsilon c} 
> \frac{a^2 - \beta_* a}{c^2 - \beta_* c} = \frac{d^2 - \beta_* d}{c^2 - \beta_* c} = \tilde \kappa (\beta_*). 
\]
The same can be shown for $\varepsilon < 0$.
Thus, $\beta_{opt} \equiv \beta_* = c - |b|$. 
The optimal convergence factor $\rho \equiv \rho_{opt}$ is then given 
by~\eqref{eqn:opt_cv}, and is obtained by evaluating $\rho_{opt}(\beta)$ in~\eqref{eq:rho_beta}
for $\beta = \beta_*$ using~\eqref{eq:kappa_beta}.  
%
Finally, from~\eqref{eqn:alpha_opt_beta}, 
we derive the optimal value of $\alpha$, given by $\alpha_{opt} \equiv \alpha_{opt} (\beta_*) = 2/(| b | c + |a|d)$.
\end{proof}

The convergence of the PSD-like iteration~\eqref{eq:psd_indef}--\eqref{eqn:pm2_opt_cond} 
follows immediately from Theorem~\ref{thm:psi_cv} and is characterized by the corollary below.  
\begin{corollary}\label{thm:pm2_cv}
Method~\eqref{eq:psd_indef}--\eqref{eqn:pm2_opt_cond} converges to the solution 
with residuals satisfying 
\begin{equation}\label{eqn:pm2_cv}
\frac{\|r^{(i+1)}\|_T}{\|r^{(i)}\|_T} \leq \frac{|ad| - |bc|}{|ad| + |bc|}.
\end{equation}
\end{corollary}
\begin{proof}
Since $\alpha^{(i)}$ and $\beta^{(i)}$ in~\eqref{eq:psd_indef}--\eqref{eqn:pm2_opt_cond} are 
such that $r^{(i+1)}$ has the smallest $T$-norm over  
$r^{(i)} + \mbox{span}\left\{ A w^{(i)}, A s^{(i)}\right\}$,
where $w^{(i)} = T r^{(i)}$ and $s^{(i)} = T A w^{(i)}$, 
\[
\|r^{(i+1)}\|_T  =  \|r^{(i)} - \beta^{(i)} Aw^{(i)} -  \alpha^{(i)} A s^{(i)}\|_T 
\leq   \|r^{(i)} - \tilde \beta A w^{(i)} - \tilde \alpha As^{(i)} \|_T, 
\]
for any $\tilde \alpha, \tilde \beta \in \mathbb{R}$.
The inequality holds for any $\tilde \alpha$ and $\tilde \beta$ and, therefore, is valid for the 
particular choice $\tilde \beta \equiv - \alpha_{opt} \beta_{opt}$ and $\tilde \alpha \equiv \alpha_{opt}$, 
where $\alpha_{opt} = 2/(| b | c + |a|d)$ and $\beta_{opt}=c - \left| b \right|$ 
are defined by~Theorem~\ref{thm:psi_cv}.
Thus,
%
\begin{equation}\label{eq:r_psd}
\|r^{(i+1)}\|_T  \leq 
\|r^{(i)} - \alpha_{opt} A l^{(i)} \|_T \equiv \|\tilde r^{(i+1)}\|_T,
\end{equation}
where $l^{(i)} = s^{(i)} - \beta_{opt} w^{(i)}$ and 
$\tilde r^{(i+1)} = r^{(i)} - \alpha_{opt} A l^{(i)}$ is
the residual after applying a step of stationary iteration~\eqref{eqn:psi_sid} with 
optimal parameters to the starting vector $x^{(i)}$. 
Then, by Theorem~\ref{thm:psi_cv}, $\|\tilde r^{(i+1)}\|_T \leq \rho_{opt} \|r^{(i)}\|_T$, 
with~$\rho_{opt}$ defined in~\eqref{eqn:opt_cv}, and the proof of the corollary follows from~\eqref{eq:r_psd}.
%
\end{proof}


If we define $\kappa = ad/bc$, then the stepwise convergence factor
in~\eqref{eqn:pm2_cv} can be written as $(\kappa - 1)/(\kappa+1)$. The 
PMINRES asymptotic convergence factor in~\eqref{eq:minres_cv} is then obtained by taking the square 
root of $\kappa$, which gives $(\sqrt{\kappa} - 1)/(\sqrt{\kappa}+1)$. This relation 
is similar to that between the PSD and PCG convergence factors for SPD systems, where $\kappa$ 
is, instead, given by the spectral condition number of $TA$. Hence,
method~\eqref{eq:psd_indef}--\eqref{eqn:pm2_opt_cond} can be viewed as a direct analogue
of PSD in the case of symmetric indefinite systems, where the preconditioner $T$ is SPD.

\subsection{The PSDI algorithm.}\label{subsec:alg}

We now describe a simple and efficient algorithm implementing
the PSD-like iteration~\eqref{eq:psd_indef}--\eqref{eqn:pm2_opt_cond}, whose 
convergence was established in the previous section.
%
Condition~\eqref{eqn:pm2_opt_cond} implies that the new residual 
$r^{(i+1)} = r^{(i)} - \beta^{(i)} A w^{(i)} - \alpha^{(i)} A s^{(i)}$ 
is $T$-orthogonal
to $\text{span}\{Aw^{(i)}, A s^{(i)}\}$, where $w^{(i)} = Tr^{(i)}$ and $s^{(i)} = TA w^{(i)}$.
Thus, at each step of method~\eqref{eq:psd_indef}, iteration parameters $\alpha^{(i)}$ 
and~$\beta^{(i)}$ 
can be determined by imposing the orthogonality constraints 
\[
(r^{(i+1)}, Aw^{(i)})_T = 0 \ \mbox{and} \ (r^{(i+1)}, A s^{(i)} )_T = 0, 
\]
which is equivalent to solving a 2-by-2 (least-squares) system
\begin{equation}\label{eq:ls_sys}
Z^* T Z \hat c = Z^* T r^{(i)}, 
\end{equation} 
where $Z = [A w^{(i)}, \; A s^{(i)}]$, and the solution is of the form
$\hat c = (\beta^{(i)} \; \alpha^{(i)})^T$. 
It is easy to check that, if $Z^* T Z$ is nonsingular,~\eqref{eq:ls_sys}
yields iteration parameters    
\begin{equation}\label{eq:it_param}
\beta^{(i)} = (\xi \nu - \mu \eta)/(\nu \mu - \eta^2), \quad \alpha^{(i)} = (\mu^2 - \xi \eta)/(\nu \mu - \eta^2),
\end{equation}
where $\xi = (w^{(i)}, Aw^{(i)})$,  $\nu = (As^{(i)},TAs^{(i)})$,  $\mu = (w^{(i)},As^{(i)})$,  and $\eta = (s^{(i)},As^{(i)})$. Moreover, since $\mbox{det} (Z^*TZ) = \nu \mu - \eta^2$, the nonsingularity of
$Z^*TZ$ guarantees that no division by zero is encountered in evaluating the expressions for 
$\alpha^{(i)}$ and $\beta^{(i)}$, and hence iteration parameters~\eqref{eq:it_param} are well-defined
in this case. Note that our definition of the iteration parameters through solution of a least-squares
problem is similar to that in the generalized conjugate gradient methods~\cite{Axelsson:80, Axelsson:87}.   

If $Z^*TZ$ is singular, then the PSD-like iteration~\eqref{eq:psd_indef}, with 
$\alpha^{(i)}$ and $\beta^{(i)}$ computed by~\eqref{eq:it_param}, breaks down due to division by zero. 
This, however, constitutes a ``happy'' break-down, which indicates that an exact solution can be 
obtained at the given step. Indeed, since $T$ is SPD, the matrix $Z^*TZ$ is singular if and only
if the columns $Aw^{(i)}$ and $As^{(i)}$ of $Z$ are linearly dependent. The latter 
implies, in particular, that $r^{(i)}$ and $ATr^{(i)}$ are collinear, in which case 
minimization~\eqref{eqn:pm2_opt_cond} yields a zero residual. The associated exact solution  
is given by $x_* = x^{(i)} + \beta^{(i)}Tr^{(i)}$, where $$\beta^{(i)} = 
(r^{(i)}, ATr^{(i)})_T/(ATr^{(i)},ATr^{(i)})_T = (w^{(i)}, A w^{(i)})/(w^{(i)}, As^{(i)}) \equiv \xi/\mu.$$

Thus, we have proved the following proposition. 

\begin{proposition}\label{prop:breakdown}
Iteration~\eqref{eq:psd_indef} with $\alpha^{(i)}$ and $\beta^{(i)}$ defined by~\eqref{eq:it_param} does not break down, provided that $w^{(i)} = Tr^{(i)}$ and $s^{(i)} = T A w^{(i)}$ are linearly independent.
If $w^{(i)}$ and $s^{(i)}$ are linearly dependent, then $x_* = x^{(i)} + \beta^{(i)} w^{(i)}$, 
where $\beta^{(i)} = (w^{(i)}, A w^{(i)})/(w^{(i)}, As^{(i)})$, is the exact solution of $Ax = f$. 
\end{proposition}

Algorithm~\ref{alg:psd1} summarizes an implementation of the PSD-like 
method~\eqref{eq:psd_indef}--\eqref{eqn:pm2_opt_cond}, which we further refer to as the PSDI algorithm. 

\begin{algorithm}[htbp]
\begin{small}
\begin{center}
  \begin{minipage}{5in}
\begin{tabular}{p{0.5in}p{4.5in}}
{\bf Input}:  &  \begin{minipage}[t]{4.0in}
The matrix $A = A^*$, a preconditioner $T = T^*>0$, the right-hand side $f$, and
                 the initial guess $x^{(0)}$;
                  \end{minipage} \\
{\bf Output}:  &  \begin{minipage}[t]{4.0in}
                 The approximate solution $x$;
                  \end{minipage}
\end{tabular}
\begin{algorithmic}[1]
\STATE $x \gets x^{(0)}$; $w \gets T(f - Ax$); 
\WHILE {convergence not reached}
  \STATE $l \gets A w$; $s \gets T l$; 
  \STATE $\xi \gets (w, l)$;
  \STATE $l \gets As$; $q \gets Tl$; 
  \STATE $\nu \gets (l,q)$; $\mu \gets (w,l)$; $\eta \gets (s,l)$;  
\IF{$\nu \mu - \eta^2 > 0$}
  \STATE $\beta \gets (\xi \nu - \mu \eta)/(\nu \mu - \eta^2)$; $\alpha \gets (\mu^2 - \xi \eta)/(\nu \mu - \eta^2)$; 
\ELSE
 \STATE $\beta \gets \xi/\eta$; $\alpha \gets 0$; \mbox{//exact solution found} 
\ENDIF 
  \STATE Update $x \gets x + \beta w + \alpha s$ and $w \gets w - \beta s - \alpha q$;
\ENDWHILE
\STATE Return $x$.
\end{algorithmic}
\end{minipage}
\end{center}
\end{small}
  \caption{A PSD-like scheme for symmetric Indefinite systems~(PSDI)}
  \label{alg:psd1}
\end{algorithm}

Each PSDI iteration performs two matrix-vector 
multiplications and two preconditioning operations. The computation of 
parameters $\alpha^{(i)}$ and $\beta^{(i)}$ requires total of 
four inner products. The number of stored vectors is equal to~five.

\subsection{PSDI vs PMINRES(2)}\label{subsec:psdi_aspects}

Algorithm~\ref{alg:psd1} is mathematically equivalent to PMINRES restarted 
after every two steps. Therefore, a possible implementation of
method~\eqref{eq:psd_indef}--\eqref{eqn:pm2_opt_cond} can be obtained by directly 
restarting any ``black box'' PMINRES solve. However, such an implementation, referred to as PMINRES(2), 
is not optimal as each restart will accrue an additional matrix-vector product and
preconditioning operation that take place at the setup phase 
to form an initial preconditioned residual vector.
By contrast, each PSDI iteration in Algorithm~\ref{alg:psd1} performs a minimal 
number of operations and gives a simple and efficient implementation 
of~\eqref{eq:psd_indef}--\eqref{eqn:pm2_opt_cond}.

\subsection{PSDI vs PMINRES}\label{subsec:psdi_pminres}
Clearly, the convergence of PSDI is generally slower than that
of PMINRES, as confirmed by bounds~\eqref{eq:minres_cv} and~\eqref{eqn:pm2_cv}.
However, in some specific situations, to~be illustrated by our numerical examples,
the reduction in computation and storage offered by PSDI (discussed below) can offset the benefit 
of a faster~convergence. 

Although PMINRES performs only one matrix-vector product and one preconditioning 
operation per step, according to~\eqref{eq:minres_cv}, 
it guarantees the residual norm reduction only after every two iterations.
Thus, both PSDI and PMINRES require two matrix-vector multiplications and
two preconditioning operations to ensure the decrease of the residual $T$-norm.
Similar to PSDI, PMINRES performs two inner products per matrix-vector multiplication,
so that the number of inner products needed for the residual reduction after 
two PMINRES steps is four.
However, PMINRES also requires an additional inner product at the setup phase prior
to the main loop; see, e.g.,~\cite[Chapter 8]{Greenbaum:97}. This extra 
work can potentially be sensible, e.g., if the total number of iterations is small 
or if the linear solve is repeatedly invoked for a sequence
of systems.       



More pronounced are memory savings. In contrast to only five vectors stored by PSDI,
a PMINRES implementation relies on at least eight vectors. Four of these vectors 
stem from the preconditioned Lanczos step, three are involved in the search direction
recurrence, and one is used to accommodate the approximate solution; 
see, e.g.,~\cite[Chapter 8]{Greenbaum:97}.     
Thus, the PSDI algorithm can be attractive in cases
where storage is limited or the memory accesses are costly.

Finally, note that if the residual $T$-norm (or the $2$-norm) 
is required to~assess the convergence, then Algorithm~\ref{alg:psd1}
should also store two additional vectors $r^{(i)} = f - Ax^{(i)}$ and $A w^{(i)}$, 
and at each iteration perform an extra inner product to evaluate the residual norm.  
However, such a residual 
norm evaluation is often unnecessary in practice, and a less expensive 
stopping rule can suffice. For~exa- mple, one can determine convergence
using the largest magnitude component of the preconditioned residual $w^{(i)}$,
which is readily available at PSDI iterations.

\subsection{PSDI vs existing schemes with comparable cost and storage}\label{subsec:psdi_other}

One may naturally wonder if PSDI provides any advantage over a number of existing 
schemes with comparable cost and storage, obtained by restarting or truncating earlier 
methods, such as preconditioned Orthomin and Orthodir~\cite{Young.Jea:80}.

As we explained in Section~\ref{sec:psi}, the preconditioned Orthomin(1) algorithm, equivalent 
to PMINRES restarted after every step, generally fails to converge when applied to symmetric 
indefinite systems with an SPD preconditioner. For $j>1$, the preconditioned Orthomin($j$),
as well as its restarted versions, are known to encounter a possible break-down, 
because zero is in the field of values of $TA$~\cite{Greenbaum:97}. By contrast, 
according to Theorem~\ref{thm:pm2_cv} and Proposition~\ref{prop:breakdown},
PSDI is guaranteed to converge and does not break down.

Note that the above discussion also applies to a somewhat less well know (preconditioned) 
Orthores algorithm~\cite{Young.Jea:80}. The latter is known to be algebraically equivalent to (preconditioned) Orthomin, 
converging if and only if Orthomin converges; see~\cite{Ashby.Gutknecht:93}. 

The situation is slightly different for the preconditioned Orthodir scheme, which is known to be 
break-down free. However, restarting preconditioned Orthodir at every step is equivalent to
preconditioned Orthomin(1) and, hence, fails to converge. Restarts after every two steps yield
an implementation that is mathematically equivalent to PSDI and PMINRES(2), but which is more
costly than both, requiring more (six versus four in PSDI) inner products per restart cycle. 
The convergence behavior of the the truncated versions, Orthodir($1$) and Orthodir($2$), is not clear.

\section{Residual minimization over a one-dimensional subspace.}\label{sec:prmm}

Let us now assume that we know the endpoints $b$ and $c$ of the intervals $\mathcal{I}$. In this case, one can fix a value 
$\beta \in (b,c)$, and consider the iterative scheme 
\begin{equation}\label{eqn:pm1}
x^{(i+1)}  =  x^{(i)} + \alpha^{(i)} l^{(i)}, \ \alpha^{(i)} = \frac{( w^{(i)}, A l^{(i)})}{(A l^{(i)},T A l^{(i)})},  \ i = 0,1,\ldots,
\end{equation} 
which updates the approximate solution by performing steps in the direction 
$l^{(i)}  =  s^{(i)} - \beta w^{(i)}$.
Here, the choice of $\alpha^{(i)}$ ensures that the new residual $r^{(i+1)} = r^{(i)} - \alpha A l^{(i)}$
has the smallest $T$-norm, i.e., 
$$\alpha^{(i)} = \underset{\alpha \in \mathbb{R}}{\operatorname{argmin}} \|r^{(i)} - \alpha A l^{(i)}\|_T.$$
%
%
The following corollary of Theorem~\ref{thm:psi_cv} guarantees that method~\eqref{eqn:pm1}
converges to the solution at a linear rate.

\begin{corollary}\label{thm:pm1_cv}
Method~\eqref{eqn:pm1} converges to the solution for any $\beta \in (b , c)$ with residuals 
satisfying 
\begin{equation}\label{eqn:pm1_rho}
\frac{\|r^{(i+1)}\|_T}{\|r^{(i)}\|_T} \leq \rho \equiv \rho_{opt}(\beta), 
\end{equation}
where $\rho_{opt}(\beta)$ is defined in~\eqref{eq:rho_beta}.
Moreover, if $\beta \equiv \beta_{opt} = c - \left| b \right|$, then~\eqref{eqn:pm2_cv}~holds. 
\end{corollary}
\begin{proof}
Since $\alpha^{(i)}$ in~\eqref{eqn:pm1} delivers the smallest residual $T$-norm, we have 
\begin{equation*}
\|r^{(i+1)}\|_T = \|r^{(i)} - \alpha^{(i)} Al^{(i)}\|_T \leq \|r^{(i)} - \tilde \alpha Al^{(i)}\|_T,
\end{equation*}
for any $\tilde \alpha \in \mathbb{R}$. Hence, the inequality also holds for 
$\tilde \alpha \equiv \alpha_{opt}(\beta)$, with $\alpha_{opt}(\beta)$ defined 
in~\eqref{eqn:alpha_opt_beta},~i.e.,
\begin{equation}\label{eq:r_beta}
\|r^{(i+1)}\|_T \leq \|r^{(i)} - \alpha_{opt}(\beta) Al^{(i)}\|_T \equiv \|\tilde r^{(i+1)}\|_T,
\end{equation}
where $\tilde r^{(i+1)} = r^{(i)} - \alpha_{opt}(\beta) A l^{(i)}$ is
the residual after applying a step of stationary iteration~\eqref{eqn:psi_sid} with 
a given $\beta \in (b,c)$ and $\alpha = \alpha_{opt}(\beta)$ to the starting vector $x^{(i)}$. 
Then, following the proof of Theorem~\ref{thm:psi_cv}, 
$\|\tilde r^{(i+1)}\|_T \leq \rho_{opt}(\beta) \|r^{(i)}\|_T$, 
with~$\rho_{opt}(\beta)$ defined in~\eqref{eq:rho_beta}, 
and bound~\eqref{eqn:pm1_rho} follows from~\eqref{eq:r_beta}.
%
%
Furthermore, by Theorem~\ref{thm:psi_cv}, if $\beta \equiv \beta_{opt} = c - |b|$ 
then $\rho_{opt}(\beta)$ turns into the optimal factor~\eqref{eqn:opt_cv} and, 
hence,~\eqref{eqn:pm2_cv} holds.   
\end{proof}

Corollary~\ref{thm:pm1_cv} suggests that the fastest convergence rate
of iteration~\eqref{eqn:pm1}, given by~\eqref{eqn:pm2_cv}, 
corresponds to $\beta = c - |b|$.  
Therefore, with this choice of~$\beta$, scheme~\eqref{eqn:pm1} can also 
be viewed as an analogue of PSD in the symmetric indefinite case.

In contrast to~\eqref{eq:psd_indef}--\eqref{eqn:pm2_opt_cond}, 
the minimization in~\eqref{eqn:pm1} is performed only over a~one-dimensional subspace. 
However,  
in order to apply the scheme, one has to come up with 
reasonable estimates for the ``inner'' endpoints $b$ and $c$.
%
For example, a trivial estimate is given by $b = c = 0$, 
which turns the method into the well known preconditioned residual norm steepest descent 
scheme~\cite{Saad:03}, 
but determining $b$ and $c$ that constitute better approximations 
to the eigenvalues $\lambda_{p}$ and $\lambda_{p+1}$ of $TA$ 
can lead to a faster convergence.   

Generally, information about the spectrum of the preconditioned matrix $TA$ is not
easy to obtain. Nevertheless, for certain problems, 
such information can be available through theoretical analysis~\cite{Silvester.Wathen:94, Wathen.Silvester:93}. 
Alternatively, one can attempt to determine the fixed 
iteration parameter empirically by trying different small values of $\beta$.
Finally, estimates on $b$ and $c$ can be obtained by applying several steps of an interior
eigenvalue solver (e.g.,~\cite{Fokkema:Sleijpen.Vorst:98, Ve.Yang.Xue:15}) 
to find a few eigenvalues of $TA$ near zero. 
For example, if a sequence of systems with the same matrix is solved, then 
such eigenvalue calculations can be performed only once during preprocessing and their 
relative cost in the overall computation can be negligible.

\subsection{The PSDI-1D algorithm.}

An implementation of method~\eqref{eqn:pm1} is given in Algorithm~\ref{alg:psd2}, which we call
the PSDI-1D algorithm.  

\begin{algorithm}[!htbp]
\begin{small}
\begin{center}
  \begin{minipage}{5in}
\begin{tabular}{p{0.5in}p{4.5in}}
{\bf Input}:  &  \begin{minipage}[t]{4.0in}
The matrix $A = A^*$, a preconditioner $T = T^*>0$, the right-hand side $f$, a parameter $b < \beta < c$, and
                 the initial guess $x^{(0)}$;
                  \end{minipage} \\
{\bf Output}:  &  \begin{minipage}[t]{4.0in}
                 The approximate solution $x$;
                  \end{minipage}
\end{tabular}
\begin{algorithmic}[1]
\STATE $x \gets x^{(0)}$; $w \gets T(f - Ax)$; 
\WHILE {convergence not reached}
  \STATE $s \gets TA w$;
  \STATE $l \gets s - \beta w$;
  \STATE $s \gets Al$; $q \gets T s$; 
  \STATE $\alpha \gets (w, s)/(s, q)$; 
  \STATE Update $x \gets x + \alpha l$ and $w \gets w - \alpha q$;
\ENDWHILE
\STATE Return $x$.
\end{algorithmic}
\end{minipage}
\end{center}
\end{small}
  \caption{A PSD-like scheme for symmetric Indefinite systems 
  with residual minimization over a 1D subspace (PSDI-1D)}
  \label{alg:psd2}
\end{algorithm}

Similar to Algorithm~\ref{alg:psd1}, each PSDI-1D iteration requires two
matrix-vector multiplications and two preconditioning operations. 
At the same time, due to the available information about the 
spectrum,
Algorithm~\ref{alg:psd2} brings the number of inner products per iteration down to two
(one per matrix-vector product), which is two times less than in PMINRES.
The number of stored vectors is five, as in Algorithm~\ref{alg:psd1}. 

\begin{table}[!htbp]
\centering
{\small
\begin{tabular}{llll}
 & PMINRES & PSDI & PSDI-1D\tabularnewline
\hline 
MatVecs/Precs & 2 & 2 & 2\tabularnewline
Inner products & 4 (+1) & 4 & 2\tabularnewline
Storage ($\#$ of vec.) & 8 & 5 & 5\tabularnewline
\end{tabular}
}
\caption{
{\small Computational and storage expenses 
of different algorithms to ensure reduction of the residual $T$-norm;
``(+1)'' denotes 
an additional inner product at the setup phase prior to the main loop.}}
\label{tbl:complexity}
\end{table}

Table~\ref{tbl:complexity} summarizes the computational and storage expenses 
of different algorithms to ensure reduction of the residual $T$-norm.     
It shows that, while generally exhibiting a slower convergence, 
the PSD-like methods need fewer inner products and storage to reduce the residual.
Therefore, if used in a proper context, the algorithms can be of practical interest
for obtaining the best performance.     

\subsection{Randomization of the search direction.}

It is common in practice that Algorithm~\ref{alg:psd2} (as well as Algorithm~\ref{alg:psd1}) rapidly reduces the residual $T$-norm
at a few initial iterations and then stabilizes with a slower convergence rate, resembling
the worst-case behavior given by bound~\eqref{eqn:pm1_rho} or, if $\beta = c - |b|$, 
by~\eqref{eqn:pm2_cv}. A possible way to break this scenario, and hence speed up the convergence,
is to exploit the freedom on the choice of $\beta \in (b,c)$ by randomly varying 
the parameter in the course of iterations. As we explain below, and demonstrate in the numerical
examples of the next section, this simple randomization of $\beta$, and therefore
of the search direction $l^{(i)} = s^{(i)} - \beta w^{(i)}$, can lead to a substantial acceleration
of the method's convergence.

At each step of method~\eqref{eqn:pm1}, the error transformation can be written as 
\begin{equation}\label{eq:err}
e^{(i+1)} = (I - \alpha^{(i)} S_{\beta})e^{(i)}, \quad S_{\beta} = (TA - \beta I)TA,
\end{equation}
which corresponds to a step of the power method with respect to the transition 
matrix $I - \alpha^{(i)} S_{\beta}$.
This step emphasizes the error component in the direction of the eigenvector associated with the 
largest, in the absolute value, eigenvalue of $I - \alpha^{(i)} S_{\beta}$.
Since the choice $\beta \in (b,c)$ ensures that all eigenvalues of $S_{\beta}$ are positive, 
regardless of $\alpha^{(i)}$, the largest modulus eigenvalue of the transition matrix
is given either by $1 - \alpha^{(i)}\mu_{\min}$ or by $1 - \alpha^{(i)}\mu_{\max}$, where
$\mu_{\min}$ and $\mu_{\max}$ are the smallest and largest eigenvalues of $S_\beta$, 
with the corresponding eigenvectors $v_{\min}$ and $v_{\max}$.
%

Thus, after repeatedly performing transformation~\eqref{eq:err}, the error will be dominated by components in the direction
of either $v_{\min}$ or $v_{\max}$, or a combination of the two. Hence, a potentially slow convergence of~\eqref{eqn:pm1}
can be attributed to the difficulty in damping these two components of the error. 

Since the eigenvalues of $S_{\beta}$ are obtained from those of $TA$ via the quadratic transformation, 
$\mu_{\min} = \min_{\lambda \in \{\lambda_p, \lambda_{p+1}\} } (\lambda^2 - \beta \lambda)$ and 
$\mu_{\max} = \max_{\lambda \in \{\lambda_1, \lambda_{n}\} } (\lambda^2 - \beta \lambda)$. 
Therefore, depending on the choice of $\beta$, $v_{\min}$ is given by $v_p$ or $v_{p+1}$, 
and $v_{\max}$ corresponds to $v_1$ or $v_n$, where $v_1$, $v_{p}$, $v_{p+1}$, and $v_n$ are 
the eigenvectors of $TA$ associated with the eigenvalues $\lambda_1$, $\lambda_{p}$, $\lambda_{p+1}$, and $\lambda_n$,
respectively.

Now, without loss of generality, suppose that the parameter $\beta$ yields $\mu_{\min} = \lambda_{p}^2 - \beta \lambda_{p}$
and $\mu_{\max} = \lambda_1^2 - \beta \lambda_1$, so that after a number of steps the error is dominated by the eigenvectors 
$v_{\min} = v_{p}$ or/and $v_{\max} = v_1$. At this point, 
let us assume that we can alter the parameter $\beta$ in such a way that $\mu_{\min}$ and $\mu_{\max}$ change to 
$\lambda_{p+1}^2 - \beta \lambda_{p+1}$ and $\lambda_{n}^2 - \beta \lambda_{n}$, respectively. 
(The change of $\mu_{\min}$ can always be achieved by modifying $\beta$, whereas the change of $\mu_{\max}$ depends on the 
location of the $TA$'s spectrum.) 
As a result, after the update of $\beta$, the error transformation~\eqref{eq:err} will emphasize the components in the direction of 
$v_{p+1}$ or $v_n$, and efficiently reduce the components in the problematic directions $v_{p}$ and $v_{1}$ that have
been dominant in the error's representation. Thus, even though the optimal convergence rate is given by $\beta = c - |b|$,
varying $\beta$ can potentially improve the convergence through
the implicit damping of the slowly vanishing error components.


   
A simple approach to systematically vary $\beta$ is to randomly generate a value from the interval 
$(b,c)$ at every iteration, i.e., set $\beta \equiv \beta_{\xi} = b + (c-b)\xi$, where $\xi$
is a random variable uniformly distributed in $(0,1)$.
Clearly, in this case, the optimal bound~\eqref{eqn:pm2_cv} no longer holds, 
however, the stepwise decrease of the residual norm is guaranteed by Corollary~\ref{thm:pm1_cv}. 
Although this reduction can be very small at certain iterations, overall,
the randomization of $\beta$ can lead to a noticeably faster convergence, as demonstrated in our examples 
of the next~section.   

\section{Examples}\label{sec:num}

In this section, we demonstrate the convergence behavior of the introduced schemes on 
several simple examples 
that admit SPD preconditioning. 
Our goal is two-fold. First, we would like to
%
illustrate the convergence bound~\eqref{eqn:pm2_cv}, as
well as show the impact of the simple randomization strategy of Algorithm~\ref{alg:psd2} 
on the convergence rate.
Second, we outline  
situations where using the PSD-like methods can represent a reasonable alternative 
to applying the optimal PMINRES.
As we shall see, such situations can occur in the cases where 
only a few iterations are needed to approximate the solution to the desired accuracy, e.g.,
due to a good initial guess or high preconditioning quality. 

\paragraph{Example 1}
In our first example, we consider a symmetric indefinite system coming from a 
discretization of the boundary value problem 
\begin{equation}\label{eqn:helmholtz_bvp}
-\Delta u (\mathrm x, \mathrm y) - \sigma u(\mathrm x,\mathrm y)  =  f(\mathrm x,\mathrm y), \  (\mathrm x,\mathrm y) \in \Omega  = (0,1)\times(0,1), \ u|_\Gamma = 0,    
\end{equation} 
where $\displaystyle \Delta = \partial^2/\partial \mathrm x ^2 + \partial^2 /\partial \mathrm y^2$ is the Laplace operator 
and $\Gamma$ denotes the boundary of the domain $\Omega$, given by a unit square. 
This problem is the Helmholtz equation with Dirichlet boundary conditions, where $\sigma > 0$ is a wave number.

Discretization of~\eqref{eqn:helmholtz_bvp}, using the standard
5-point finite difference stencil, results in a linear system $(L - \sigma I)x = f$, 
where $L$ represents the discrete Laplacian.
Since $L$ is SPD, the choice of a sufficiently large $\sigma$ introduces negative eigenvalues into the shifted problem, making the matrix
$L - \sigma I$ indefinite. If the degree of indefiniteness is not too high, a simple option to define an SPD preconditioner 
for $(L - \sigma I)x = f$ is given by $T = L^{-1}$~\cite{Bayliss.Goldstein.Turkel:83}. 
Below, we use such $T$ as an SPD preconditioner 
for the PSD-like schemes and the PMINRES algorithm.
The right-hand side $f$ and the initial guess $x^{(0)}$ are randomly chosen.

\begin{figure}[h] 
 \begin{center} 	 
 \includegraphics[width = 6cm]{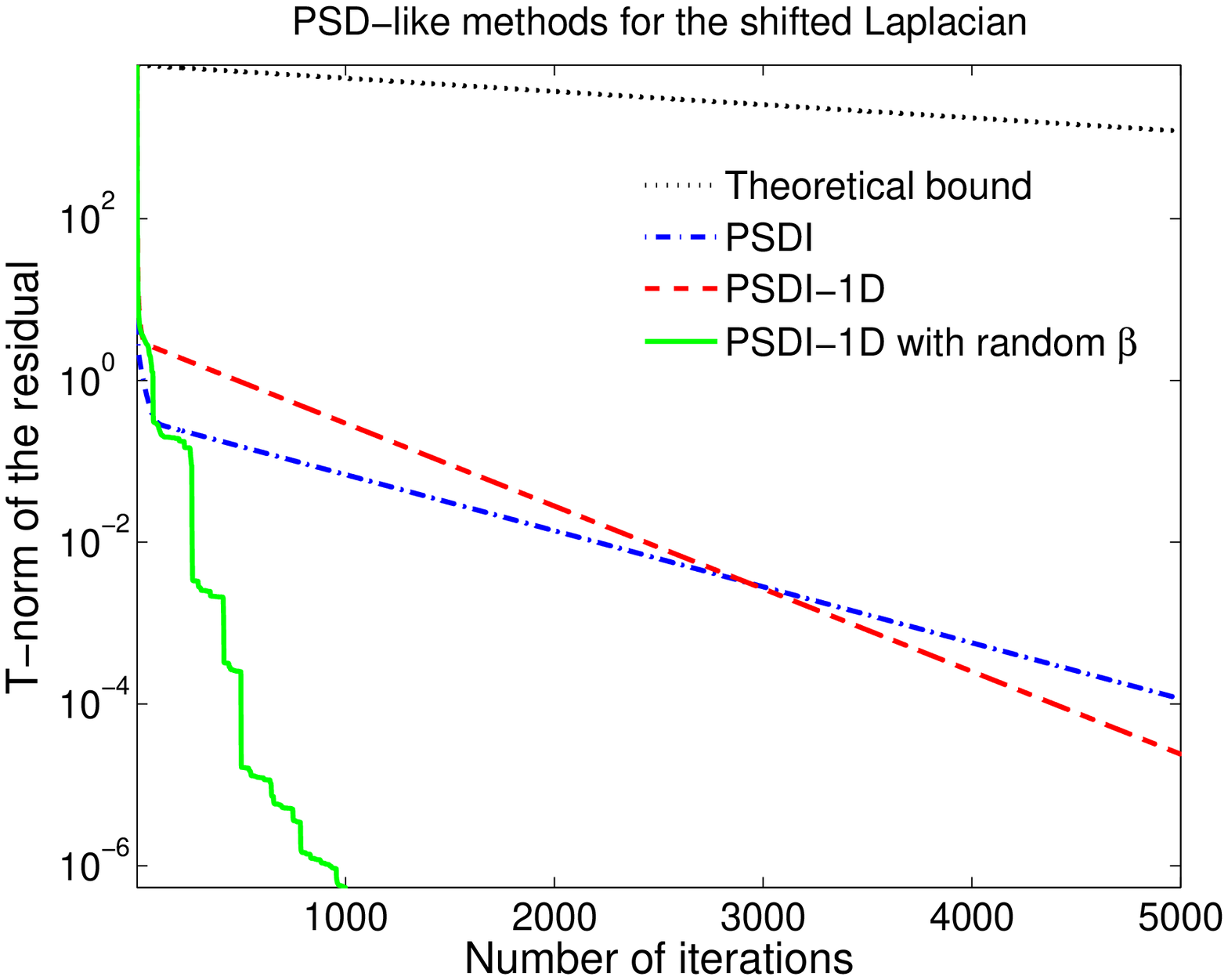} 
 \includegraphics[width = 6cm]{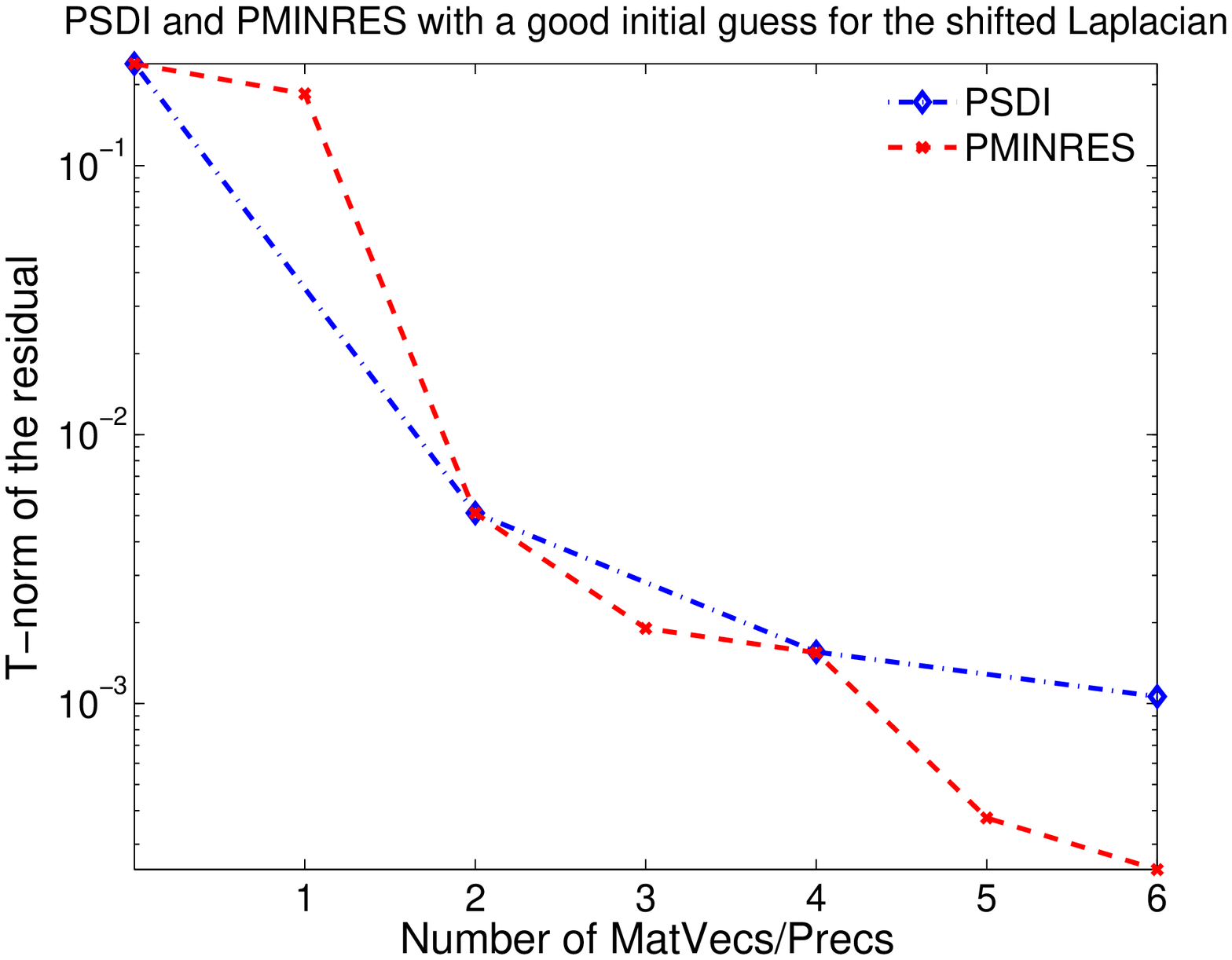} 
 \end{center} 
 \caption{\small{Convergence of different solution schemes for the shifted Laplacian system $(L - \sigma I)x = f$ 
with $\sigma = 100$ and $T = L^{-1}$; $n = 3,969$. \textit{Left:} Comparison of the PSD-like schemes.
\textit{Right:} Comparison of PSDI and PMINRES for a good initial guess.}}\label{fig:lapl} 
\end{figure}

In particular, we let $\sigma = 100$ and consider the shifted Laplacian problem of size $n = 3,969$.     
Then, if $T = L^{-1}$, the preconditioned matrix $T(L - \sigma I)$ has 6 negative eigenvalues, with
$\lambda_1 \approx -4.0671$, $\lambda_6 \approx - 0.0149$, $\lambda_7 \approx 0.2194$, 
and $\lambda_n = 0.9939$.
Thus, the interval $\mathcal{I} = [a,b]\cup[c,d]$, containing the spectrum of $T(L - \sigma I)$, 
can be defined by $a = \lambda_1$, $b = \lambda_6$, $c = \lambda_7$, and $d = 4.2716$, 
where the choice of~$d$ ensures that $[a,b]$ and $[d,c]$ are of the same length. 
This information allows us to calculate convergence
bound~\eqref{eqn:pm2_cv} and set the parameter $\beta$ in Algorithm~\ref{alg:psd2} to the optimal 
value $c - |b|$.
The generation of $\beta$ in the randomized version of Algorithm~\ref{alg:psd2} is performed 
with respect to the interval~$(b, c)$. 

We note that the question of constructing efficient SPD preconditioners for Helmholtz problems 
is not in the scope of this paper, and the choice $T = L^{-1}$ is motivated mainly by simplicity 
of presentation, allowing to keep focus on the presented PSDI iterative scheme rather than on 
preconditioning issues. A stronger SPD preconditioner for this model problem can be found in~\cite{thesis,Ve.Kn:13}. 


The convergence of the PSD-like schemes is demonstrated in Figure~\ref{fig:lapl} (left). 
The figure shows that bound~\eqref{eqn:pm2_cv} is descriptive.  
It reflects well the convergence rate of PSDI and (non-randomized) PSDI-1D throughout the whole run, 
except for a few initial steps where the residual norms are reduced faster in practice.  
Note that PSDI-1D has a slightly faster convergence than PSDI, which demonstrates
that minimizing the residual over a 1D subspace does not necessarily yield a 
slower convergence compared to the 2D minimization of Algorithm~\ref{alg:psd2}. 
%
We also observe a significant acceleration of the convergence if a random~$\beta$ 
is used within PSDI-1D.
Remarkably, the speedup appears at no additional cost and is a consequence solely of the ``chaotic'' choice of the descent direction.   

Next, we consider a specific setting, where the initial guess
is already close to the solution and only low to moderate accuracy of the targeted approximation is 
wanted. In this case, if the preconditioning quality is sufficiently high, only a few steps of a
linear solver should be performed. 

The convergence of PSDI and PMINRES for such a situation is compared in Figure~\ref{fig:lapl} (right). 
Namely, we compute the exact solution of $(L - \sigma I) x = f$ and perturb it using a
random vector with small entries distributed uniformly on $[0,10^{-4}]$. We then apply three 
steps of PSDI and six steps of PMINRES and track the reduction of the residual $T$-norm at the
few initial iterations.
%
Since each PSDI iteration requires twice as many matrix-vector products and 
preconditioner applications compared to the PMINRES step, instead of the iteration count,
we show the convergence rate with respect to the number of matrix-vector multiplications
(MatVecs) or preconditioning operations (Precs). 

Figure~\ref{fig:lapl} (right) shows that both algorithms require the same number
of MatVecs/Precs to achieve the reduction of the residual $T$-norm by two orders 
of magnitude, i.e., the residual $T$-norms after two PSDI steps and four PMINRES steps
are identical. At the same time, as has been previously discussed, 
PSDI performs slightly less inner products and requires less memory. 
Hence, in the given context, if the goal is to improve the solution accuracy by only 
a few orders of magnitude, PSDI can be used as an alternative to PMINRES. 

However, if higher accuracies are wanted, which requires additional iterations, 
then PMINRES, as an optimal method, is clearly more suitable. For example,
as seen in Figure~\ref{fig:lapl} (right), its convergence becomes noticeably 
faster then that of PSDI starting from the fifth iteration. 
Note that the convergence of PSDI-1D
at the initial steps,
with both optimal and random choice of $\beta$, 
was not as rapid compared to PSDI and PMINRES. Therefore, we do not report the corresponding 
runs in the figure.

\paragraph{Example 2}
Our second example concerns a saddle point system, arising in the context of PDE-constrained optimization. 
Here, the solution of the optimal control problem
\[
\min_{u,f}\frac{1}{2}\|u - \hat u\|^2 + \tau \|f\|^2, 
\]
with the constraint that 
\[
-\Delta u  =  f \ \mbox{in} \ \Omega, \ u|_\Gamma = g,
\]
results, after the finite element discretization, in the symmetric indefinite system with the matrix
\begin{equation}\label{eq:saddle}
A = \left[
\begin{array}{rrr}
2 \tau M &  0  & -M \\
0        &  M  & K \\
-M       &  K  & 0
\end{array}
\right],
\end{equation}
where $K$ and $M$ are the SPD stiffness and mass matrices, respectively; 
see, e.g.,~\cite{Rees.Dollar.Wathen:09}. 
In particular, we choose $\Omega = (0,1)\times(0,1)$, $\tau = 10^{-2}$, 
$\hat u =  (2 \mathrm x -1)^2 (2 \mathrm y - 1)^2$ over 
$(0,\frac{1}{2})\times(0,\frac{1}{2})$ and 0 elsewhere, and use $Q_1$ finite elements to obtain the saddle point linear system of size $n = 2,883$. 
Exactly the same example was considered by Wathen and Rees~\cite{Wathen.Rees:09}, 
whereto we refer the reader for more details.

An efficient SPD preconditioner for~\eqref{eq:saddle}, proposed in~\cite{Rees.Dollar.Wathen:09}, has a block-diagonal form,
and is given by 
\begin{equation}\label{eq:saddle_T}
T = \left[
\begin{array}{rrr}
\frac{1}{2\tau} \tilde M^{-1} &  0  & 0 \\
0        &  \tilde M^{-1}  & 0 \\
0       &  0  & \tilde K^{*-1} M \tilde K^{-1}
\end{array}
\right],
\end{equation}
where $\tilde K$ and $\tilde M$ are approximations to $K$ and $M$, respectively.
In our test, we approximate $\tilde K$ and $\tilde M$ using incomplete Cholesky factorization 
with drop tolerance $10^{-3}$, so that $\tilde K^{-1}$ and $\tilde M^{-1}$ 
correspond to the successive triangular solves with the respective incomplete 
Cholesky factors. 
In this case, the spectrum of the preconditioned matrix $TA$ is enclosed into
the pair of equal-sized intervals $[-1.0108, -0.3096]$ and $[1; 1.7012]$.

\begin{figure}[h] 
 \begin{center} 	 
 \includegraphics[width = 6cm]{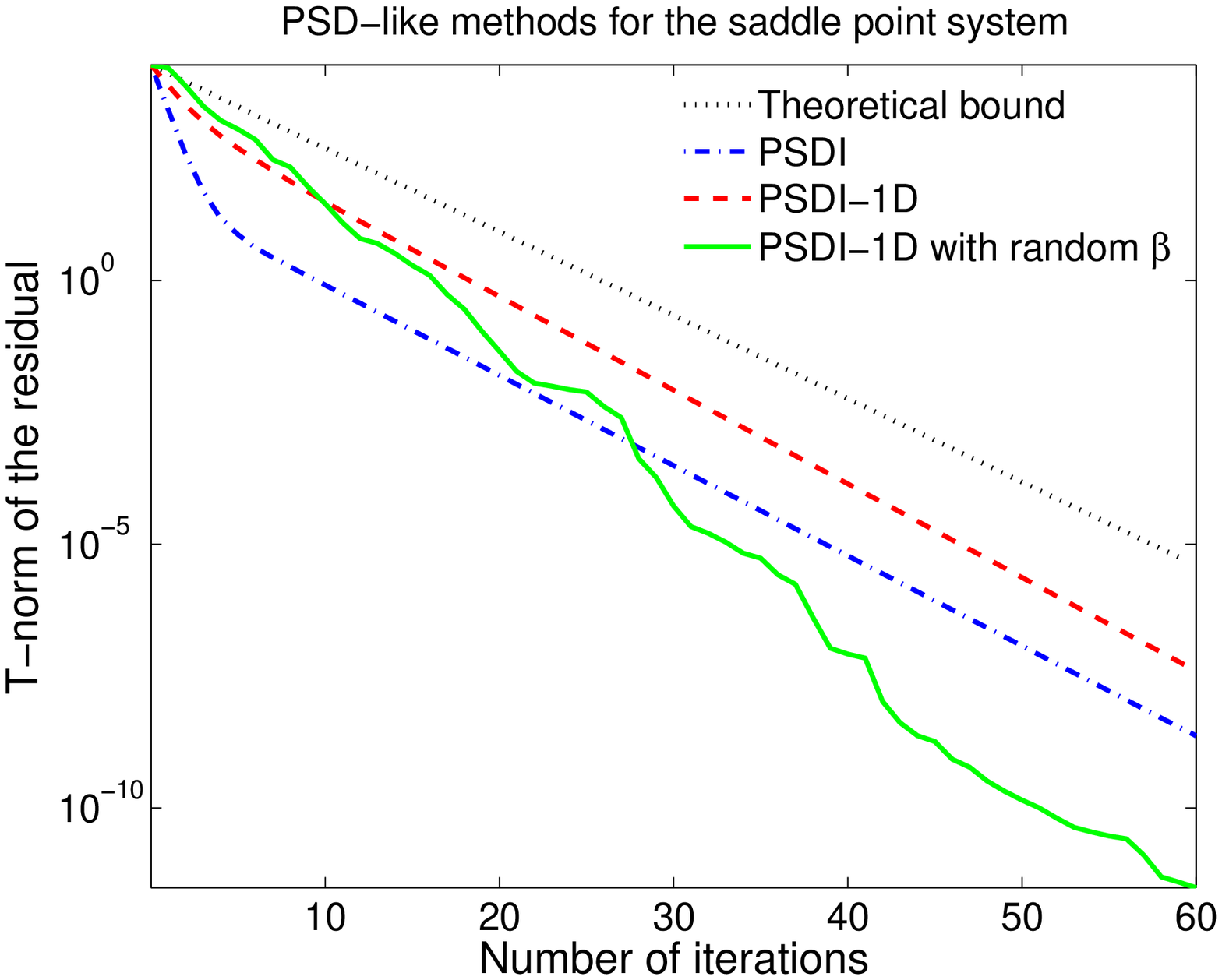} 
 \includegraphics[width = 6cm]{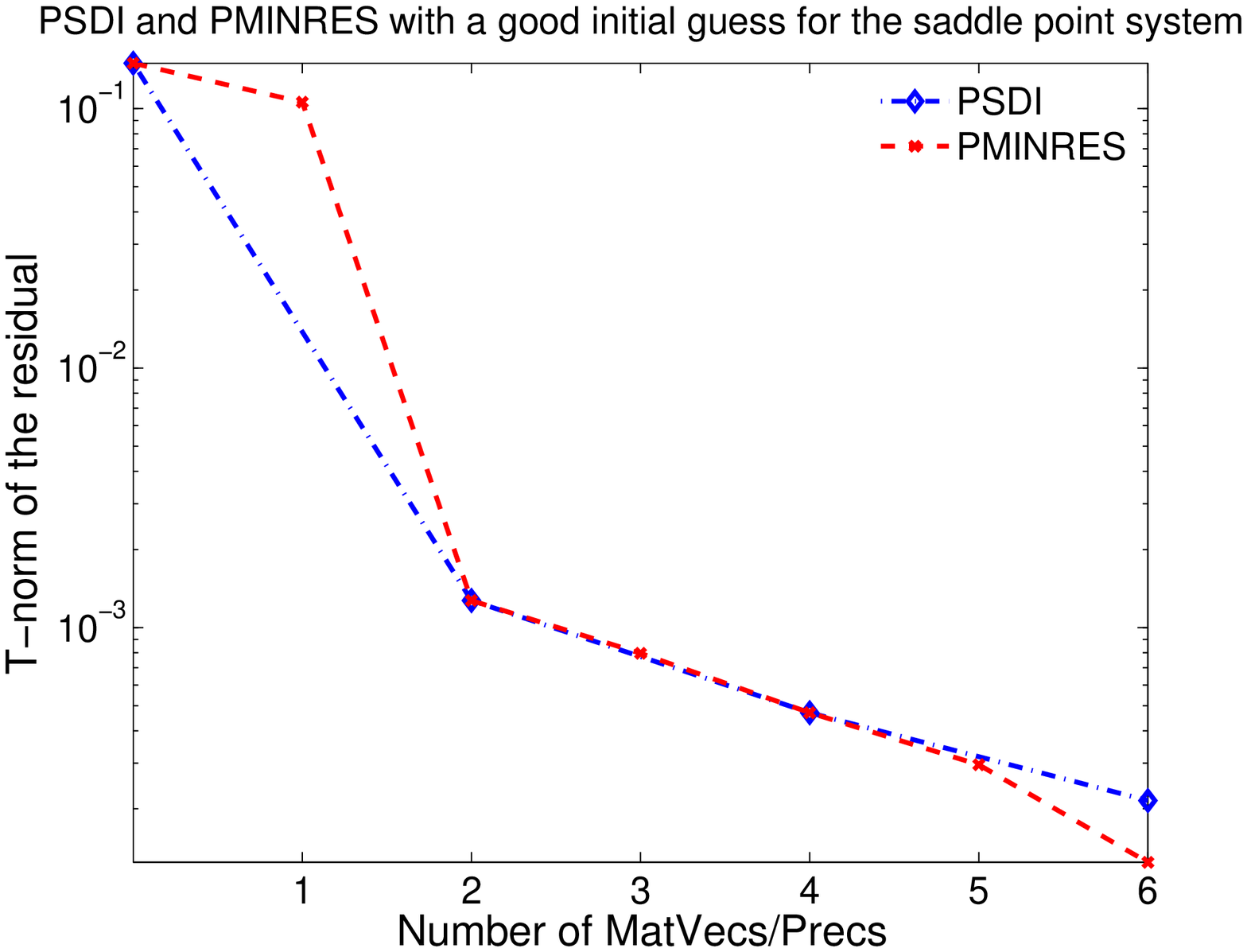} 
 \end{center} 
\caption{\small{Convergence of different solution schemes for a saddle point system with $A$
in~\eqref{eq:saddle} and $T$ in~\eqref{eq:saddle_T}; $n = 2,883$. 
\textit{Left:} Comparison of the PSD-like schemes.
\textit{Right:} Comparison of PSDI and PMINRES for a good initial guess.}}
\label{fig:saddle} 
\end{figure}

In Figure~\ref{fig:saddle} (left), we demonstrate the runs of the 
PSD-like methods for system~\eqref{eq:saddle},
with randomly chosen right-hand side and initial guess vectors. 
The parameter $\beta$ in 
PSDI-1D is set to the optimal value $c - |b|$, and the sampling of $\beta$ in the randomized version 
is performed over the interval $(b,c)$. 
%
As in the the previous example, it can be seen that bound~\eqref{eqn:pm2_cv} captures well
the actual convergence of PSDI and PSDI-1D, and the convergence rates of both schemes are
comparable in practice. The suggested randomization strategy, again, speeds up the convergence for PSDI-1D.  

Let us note that the convergence of the randomized PSDI-1D depends on the way random values $\beta$
are generated. In particular, using inappropriate probability distributions can slow down the 
convergence. On the contrary, one can expect to accelerate the method by suitably 
defining probability distribution.

\begin{figure}[h] 
 \begin{center} 	 
 \includegraphics[width = 6cm]{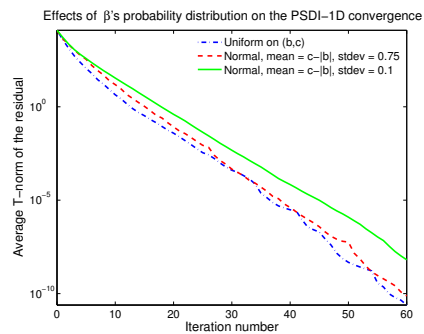} 
 \end{center} 
\caption{\small{Convergence of PSDI-1D with values of $\beta$ drawn from different distributions
for a saddle point system with $A$ in~\eqref{eq:saddle} and $T$ in~\eqref{eq:saddle_T}; $n = 2,883$.} 
}
\label{fig:distr} 
\end{figure}

This point is demonstrated in Figure~\ref{fig:distr}, which 
compares convergence of~PSDI-1D for values of $\beta$ drawn from different distributions.    
In the figure, we plot averaged (after 100 runs) residual norms produced by PSDI-1D,
where $\beta$ is either uniformly distributed on $(b,c)$ (as before), or drawn  
from the normal distribution with mean at the optimal value $c-|b|$ and~standard~deviations~$0.1$~and~$0.75$. 

One can see that a slower convergence is obtained if $\beta$ is normally distributed  
with standard deviation $0.1$, in which case the method closer resembles the deterministic version with
the optimal $\beta$. 
At the same time, increasing standard deviation to $0.75$ removes this effect, 
resulting in the convergence comparable to the case with the uniform distribution.

Finally, Figure~\ref{fig:saddle} (right) compares PSDI and the optimal PMINRES for the case where a good
initial guess is available and both methods perform only a few iterations to reduce the residual
$T$-norm by several orders of magnitude.  
Similar to the previous example, we define the initial guess by 
perturbing the exact solution with a
random vector whose entries are uniformly distributed on $[0,10^{-4}]$. 
The figure demonstrates that, at the initial iterations, the convergence of 
PSDI is comparable to that of PMINRES. However, PSDI requires less computations
and memory, and hence can be preferable to PMINRES in this type of situation.   

\paragraph{Example 3}
Another context which gives rise to symmetric indefinite systems 
is related to the interior eigenvalue calculations using inexact shift-and-invert, 
or preconditioned, eigenvalue solvers, e.g.,~\cite{Morgan:91, Ve.Yang.Xue:15}. 
In this setting, one seeks to compute an eigenpair 
$(\lambda,v)$ of a matrix $A$ that is closest to a given target $\sigma$.
At each iteration, such eigenvalue solvers require an approximate solution of the linear 
system of the form $(A - \sigma I) w = r$, where $r$ is the eigenresidual.

If a good preconditioner $T \approx (A - \sigma I)^{-1}$ is at hand, then $w$ can be defined
as $Tr$. However, in certain cases, the quality of $T$ is insufficient to ensure
a robust convergence . In this situation, instead, one can run 
several steps of an iterative linear solver applied to the symmetric indefinite system
$(A - \sigma I) w = r$ with  $T$ as a preconditioner, 
and set $w$ to the resulting approximate solution. 
In particular,
if $T$ is SPD, then the approximate solution of $(A - \sigma I) w = r$ can be 
computed either using PMINRES or one of the 
PSD-like methods introduced in this work. 

Let us consider a matrix $A$ coming from the plane wave discretization
of the Hamiltonian operator for the Si2H4 molecule ($n = 949$) 
in the framework of the Kohn-Sham Density Functional theory,
generated using the KSSOLV package~\cite{kssolv:09}. We would like to find an eigenpair
corresponding to the eigenvalue closest to the energy shift $\sigma=0.2$ using
the Davidson method with the harmonic Rayleigh--Ritz projection~\cite{Morgan:91}.
The given target $\sigma$ points to the $8$th eigenpair of $A$
associated with $\lambda = 0.1966$. 
Note that $A$ is complex Hermitian in this test,  
for which case all the results of this paper straightforwardly apply, 
though stated for the real symmetric matrices. The initial guess for the 
eigensolver is fixed to the first column of the identity matrix; the PSDI and 
PMINRES iterations start with the zero initial guess. 

A traditional choice of $T$ for this type of computation is the 
Teter--Payne--Allan preconditioner~\cite{Teter.Payne.Allan:89}, which is given by an SPD diagonal 
matrix. However, a direct use of $T$ to define the Davidson's expansion vectors 
$Tr$ may not provide a reliable convergence.
In particular, this is the case in our example, where the method 
converges to a wrong eigenpair. Therefore, in order to restore the convergence, 
as a preconditioner for the Davidson method, we use several steps of PMINRES and PSDI 
applied to $(A - \sigma I) w = r$, with $T$ being a preconditioner for the linear solve.

\begin{figure}[h] 
 \begin{center} 	 
 \includegraphics[width = 6cm]{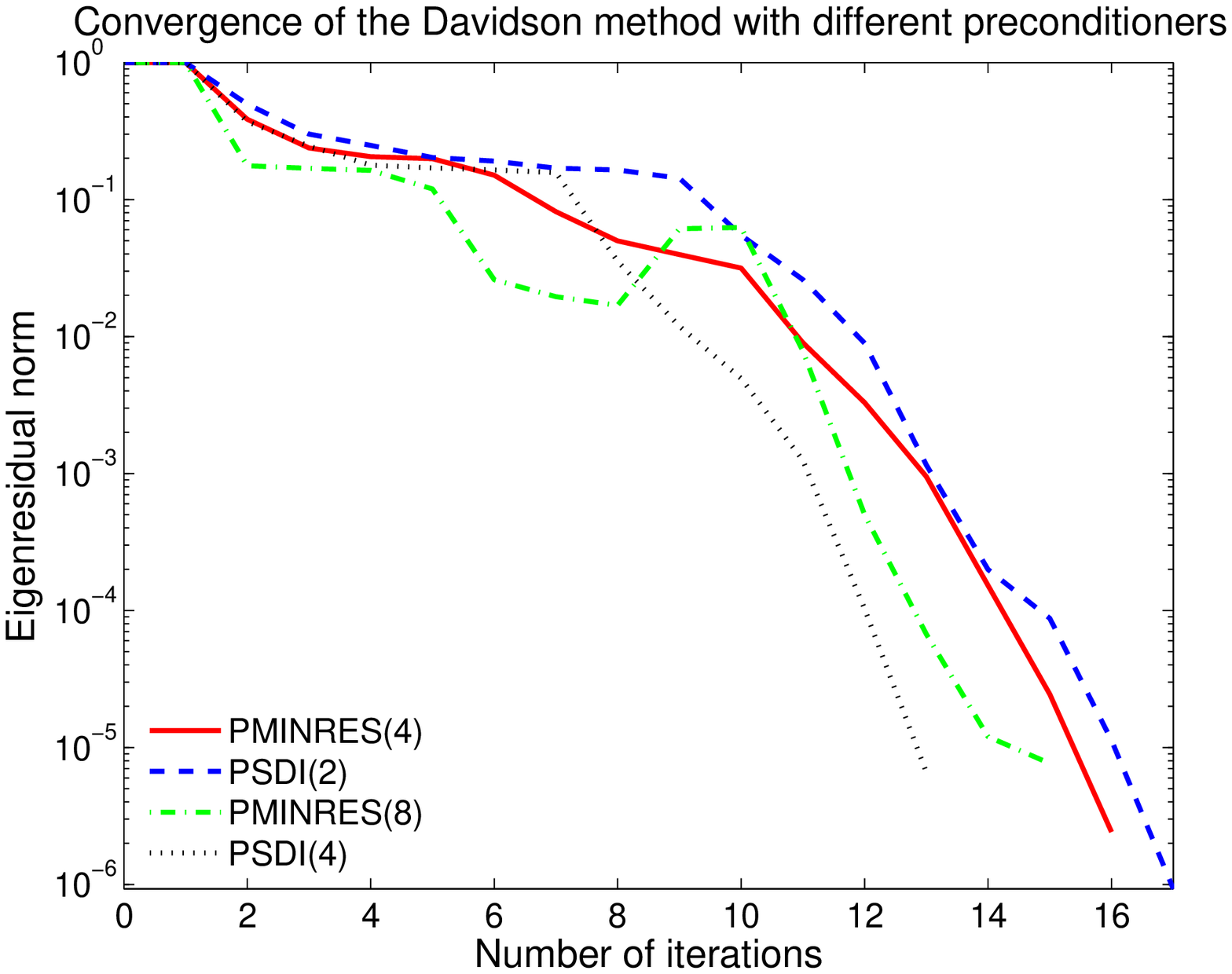} 
 \includegraphics[width = 6cm]{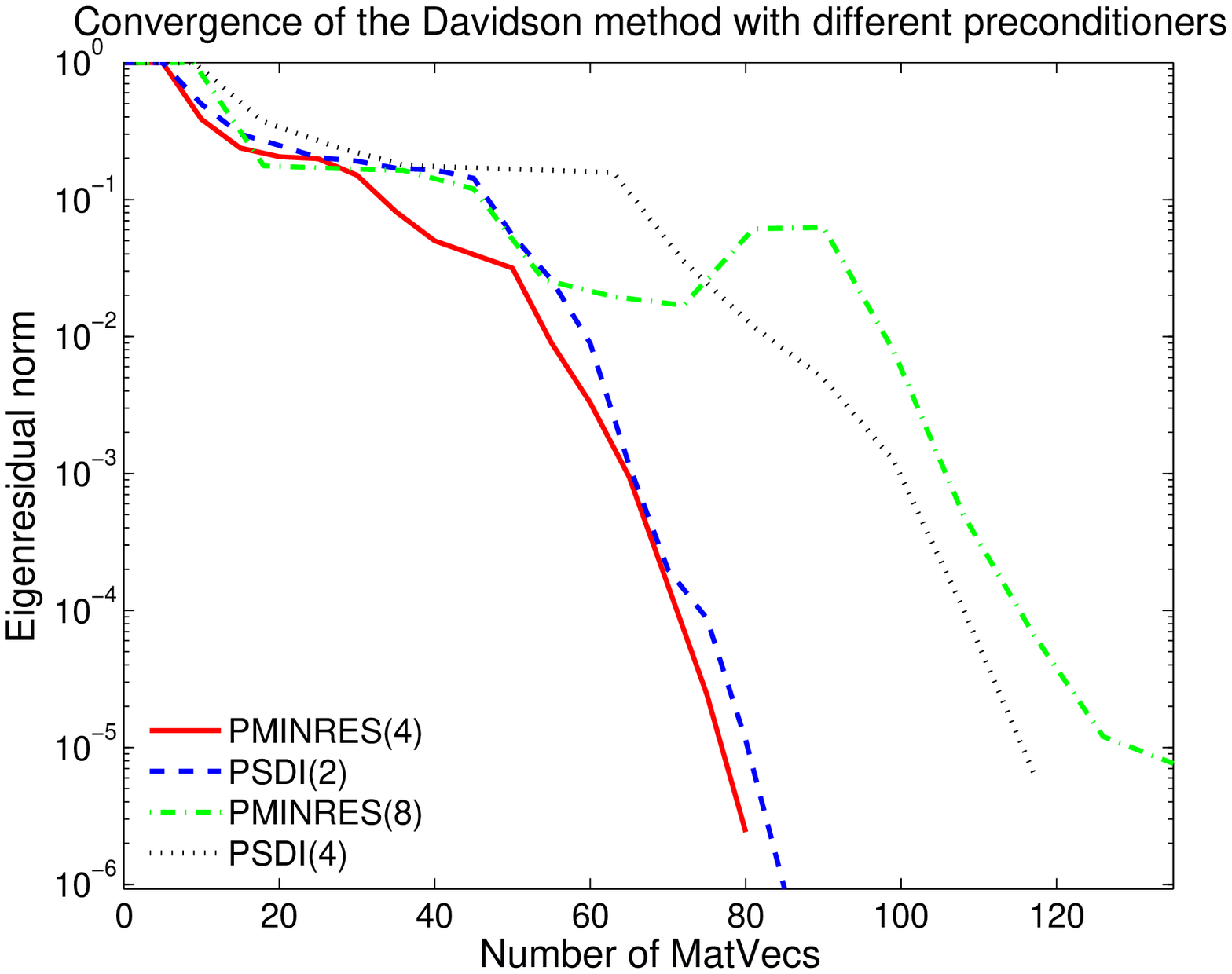} 
 \end{center} 
\caption{\small{Convergence of the Davidson method, with respect to the 
number of iterations (left) and MatVecs (right), to the eigenpair associated
with the eigenvalue closest to $\sigma = 0.2$ of the Hamiltonian matrix for the 
Si2H4 system. PSDI($t$) and PMINRES($t$) denote preconditioning options based on
$t$ steps of the corresponding linear solver.}}
\label{fig:davidson} 
\end{figure}

Figure~\ref{fig:davidson} (left) shows that the convergence to the correct eigenpair
can be recovered with 2 steps of PSDI 
and 4 steps of PMINRES 
used as a preconditioner for the Davidson method.  
In this case, the convergence of the 
PSDI-preconditioned eigensolver is similar to that
of preconditioned with PMINRES. However, the former requires less inner products and storage;
see Table~\ref{tbl:complexity}. Note that doubling the number of PSDI and PMINRES
steps 
slightly reduces the eigensolver's 
iteration count, whereas the convergence 
remains identical
for both preconditioning options.  


In Figure~\ref{fig:davidson} (right), we consider the change of the eigenresidual norm with 
respect to the number of matrix-vector products, which includes 
MatVecs generated
at the ``inner'' PSDI or PMINRES iterations as well as those produced by the ``outer''
Davidson steps. The figure demonstrates that increasing the number
of PSDI or PMINRES iterations may be counterproductive, even though the preconditioning quality
improves. As a result, we arrive at the framework where only a few steps of a linear solver
are needed, in which case the use of the PSD-like methods can represent a reasonable alternative.
to PMINRES.


\section{Conclusions}\label{sec:concl}
 
The paper presents a thorough description of the PSD-like methods for 
symmetric indefinite systems, where the preconditioner is SPD. 
Several variants of such methods are discussed and the corresponding
convergence bound is proved. This completes the existing theory for the SPD 
linear systems, expanding it to the indefinite case. Because of the slower
convergence rate, the presented PSD-like methods cannot generally be 
regarded as a substititute for the optimal PMINRES algorithm. However,
we demonstrate that for certain cases, where only a few steps of a linear
solver are needed, the PSD-like schemes can constitute an economical 
alternative.

\section*{Acknowledgements.}
The authors are thankful to Dr.~Tyrone Rees for sharing test matrix~\eqref{eq:saddle} 
for the saddle point system in Example~2. The authors also thank the anonymous referee
whose comments helped to significantly improve this manuscript. 

\bibliography{sys,plmr}

\end{document}